\documentclass[a4paper,10pt]{amsart}
\usepackage{amssymb,amscd,amsmath}

\theoremstyle{plain}
\newtheorem{thm}{Theorem}[section]

\newtheorem{lem}[thm]{Lemma}
\newtheorem{cor}[thm]{Corollary}

\theoremstyle{definition}

\theoremstyle{remark}
\newtheorem{rem}{Remark}[section]

\newcommand{\forme}[1]{}

\title{On wreath products of cyclic schemes}

\author[K.~Kim]{Kijung Kim}
\address{Department of Mathematics, POSTECH, Pohang 790-785, Republic of Korea}
\email{kijung@postech.ac.kr}

\date{\today}

\begin{document}
\maketitle

\begin{abstract}
We describe explicitly the algebraic structure of the Terwilliger algebra of wreath products of cyclic schemes.
\end{abstract}


\section{Introduction}\label{sec:intro}
The Terwilliger algebra was first introduced in \cite{terwilliger} as new tool for commutative association schemes.
In general, this algebra is non-commutative, finite dimensional and semi-simple $\mathbb{C}$-algebra.
In \cite{song} G. Bhattacharyya, S.Y. Song and R. Tanaka began to study the wreath product of one-class association schemes.
It was shown that all irreducible modules except for the primary module of the Terwilliger algebra are one-dimensional.
Moreover, in \cite{rie} R. Tanaka proved that the only class of association schemes coming as the wreath product of one-class association schemes and group schemes of finite abelian groups satisfies this property.
Recently, in \cite{song3} S.Y. Song and B. Xu gave a complete structural description of the Bose-Mesner algebra and Terwilliger algebra for wreath products of one-class association schemes.
In this paper, we describe explicitly the algebraic structure of the Terwilliger algebra of wreath products of cyclic schemes.

The remainder of this paper is organized as follows.
In Section~\ref{sec:pre}, we review the notation and basic results on association schemes and the Terwilliger algebra.
In Section~\ref{sec:main}, we show that $ C_{p_1} \wr C_{p_2} \wr \cdots \wr C_{p_d} $ is triply-regular.
In Section~\ref{sec:full}, we determine the structure of the Terwilliger algebra of $ C_{p_1} \wr C_{p_2} \wr \cdots \wr C_{p_d} $ by central primitive idempotents.

\section{Preliminaries}\label{sec:pre}
In this section, we prepare necessary notation and results about association schemes and their Terwilliger algebras.
For further information, the reader is referred to \cite{bannai}, \cite{bcn}, \cite{godsil} and \cite{terwilliger}.

\subsection{Association schemes}\label{sec:as}

Let $X$ denote a nonempty finite set. Let $Mat_{|X|}(\mathbb{C})$ denote the $\mathbb{C}$-algebra consisting of all matrices whose rows and columns
are indexed by $X$ and whose entries are in $\mathbb{C}$.

Let $R_0, R_1, \dotsc , R_d$ be nonempty subsets of $X \times X$.
Let $A_i$ denote the matrix in $Mat_{|X|}(\mathbb{C})$ with $xy$ entry
\[(A_i)_{xy}  = \left\{
                    \begin{array}{ll}
                     1 & \hbox{if $(x,y) \in R_i$;} \\
                     0 & \hbox{otherwise.}
                    \end{array}
                  \right.\]
It is called \textit{adjacency} matrix of $R_i$.
We denote the transpose of $A_i$ by $A_i ^t$, the identity matrix by $I$ and all-ones matrix by $J$.

We say that $\chi = (X, \{ R_i \}_{ 0 \leq i \leq d} )$ is a \textit{d-class association scheme} of order $|X|$
if the following hold:
\begin{enumerate}
\item[(1)] $A_0 = I$;
\item[(2)] $A_0 + A_1 +\cdots + A_d=J$;
\item[(3)] $A_i ^t = A_{i'}$ for some $0 \leq i' \leq d$;
\item[(4)] For each $0 \leq i,j \leq d$, $A_i A_j = \sum _{h=0} ^d  p_{ij} ^h A_h$ for some nonnegative integers $p_{ij} ^h$.
\end{enumerate}

The constants $p_{ij} ^h$ given in (4) are called the \textit{intersection numbers} of $\chi$.
For each $A_i$ we abbreviate $p _{ii'}^0$ as $n_i$, which is called the \textit{valency} of $A_i$.
We denote $\{ y \in X \mid (x,y) \in R_i \}$ by $R_i (x)$.
Then for each $(x,y) \in R_h$, $p_{ij} ^h=|\{ z \in X \mid (x,z) \in R_i, (z,y) \in R_j \}|=|R_i(x) \cap R_{j'}(y)|$.
We say that $\chi$ is \textit{commutative} if $A_i A_j = A_j A_i$ for $0\leq i,j \leq d$.
In \cite{pamela}, a commutative association scheme $\chi$ ia called \textit{cyclic} if $\{ R_i \}_{ 0 \leq i \leq d}$ forms a cyclic group.
Moreover, if the order of a cyclic scheme $\chi$ is $n$, then we denote $\chi$ by $C_n$.
Next, we recall the concept of the wreath product and use the notation given in \cite{song2}. Let $\chi = (X, \{ R_i \}_{ 0 \leq i \leq d} )$ and $\psi = (Y, \{ S_i \}_{ 0 \leq j \leq e} )$ be association schemes of order $|X|=u$ and $|Y|=v$.
The \textit{wreath product} $\chi \wr \psi$ of $\chi$ and $\psi$ is defined on the set $X \times Y$;
but we take $Y = \{ y_1 , y_2, \dotsc , y_v \}$, and regard $X \times Y$ as the disjoint union of $v$ copies $X_1, X_2, \dotsc , X_v$
of $X$, where $X_j = X \times \{ y_j \}$.
The relations on $X_1 \cup X_2 \cup \cdots \cup X_v$ are defined by the following rules:
\begin{enumerate}
\item[(1)] For any $j$, the relations between the elements of $X_j$ are determined by the
association relations between the first coordinates in $\chi$.
\item[(2)] The relations between the elements that belong to two different sets, say $X_i$ and
$X_j$, are determined by the association relation of the second coordinates $y_i$ and
$y_j$ in $\psi$ and the relation is independent from the first coordinates.
\end{enumerate}
That is, the relations $W_0, W_1, \dotsc , W_{d+e}$ of $\chi \wr \psi$ are defined by

$W_0 = \{ ((x,y),(x,y)) \mid (x,y) \in X \times Y \}$,

$W_k = \{ ((x_1,y), (x_2,y)) \mid (x_1, x_2) \in R_k , y \in Y \}$ for $1 \leq k \leq d$ and

$W_k = \{ ((x_1, y_1),(x_2,y_2)) \mid x_1, x_2 \in X , (y_1, y_2) \in S_{k-d} \}$ for $d+1 \leq k \leq d+e$.

Then it is easy to see that $\chi \wr \psi = ( X \times Y, \{ W_k  \}_{0 \leq k \leq d+e})$ is a $(d+e)$-class association scheme.
Note $\chi \wr \psi$ is commutative if and only if $\chi$ and $\psi$ are.

Let the adjacency matrices of $\chi$ and $\psi$ be $A_0,  A_1 , \dotsc , A_d$ and $C_0, C_1, \dotsc , C_e$, respectively.
Then the adjacency matrices of $\chi \wr \psi$ are given by
\[ C_0 \otimes A_0, C_0 \otimes A_1, \dotsc , C_0 \otimes A_d, C_1 \otimes J_u , \dotsc ,C_e \otimes J_u,\]
where $\otimes$ denotes the Kronecker product, i.e., $A \otimes B := (a_{ij}B)$ of $A=(a_{ij})$ and $B$.

\subsection{Terwilliger algebras}\label{sec:algebra}
Let $\chi = (X, \{ R_i \}_{ 0 \leq i \leq d} )$ be a $d$-class association scheme.
Let $V=\mathbb{C}X = \bigoplus_{x \in X} \mathbb{C} \hat{x}$ denote the vector space over $\mathbb{C}$ consisting of column vectors whose
coordinates are indexed by $X$ and whose entries are in $\mathbb{C}$, where $\hat{x}$ denotes the column vector with $1$ in the $x$-th position and $0$ otherwise.
We observe that $Mat_{|X|}(\mathbb{C})$ acts on $V$ by left multiplication. We call $V$ the \textit{standard module}.
We endow $V$ with the Hermitian form  defined by $\langle u , v \rangle = u^t \bar{v}$ for $u, v \in  V$, where $t$ denotes
transpose and - denotes complex conjugation.

If $\chi$ is commutative, then the adjacency matrices generate a $(d+1)$-dimensional subalgebra $\mathfrak{M} =\langle A_0, A_1, \dotsc, A_d \rangle$, which is called the \textit{Bose-Mesner} algebra.

Fix a vertex $x \in X$. For each $0 \leq i \leq d$, let $E_i ^* = E_i ^*(x)$ denote the diagonal matrix in $Mat_{|X|}(\mathbb{C})$ with $yy$ entry
\[(E_i ^*)_{yy}  = \left\{
                    \begin{array}{ll}
                     1 & \hbox{if $(x,y) \in R_i$;} \\
                     0 & \hbox{otherwise.}
                    \end{array}
                  \right.\]

The $E_i ^*$'s form a basis for a subalgebra $\mathfrak{M}^*=\mathfrak{M}^*(x) =\langle E_0^*, E_1^*, \dotsc, E_d^* \rangle$.
Let $\mathcal{T}=\mathcal{T}(x)$ denote the subalgebra of $Mat_{|X|}(\mathbb{C})$ generated by $\mathfrak{M}$ and $\mathfrak{M}^*$.
We say that $\mathcal{T}$ is the \textit{Terwilliger algebra of $\chi$ with respect to $x$}.
Put $V_i ^* = E_i ^* V$. Then $V=\bigoplus_{i=0} ^d V_i ^*$.
We denote the all-ones column vector in $V$ by $\mathbf{1}$.

%Two $\mathcal{T}$-modules $W , W'$ are \textit{isomorphic} if there exists a vector space isomorphism $\phi : W \rightarrow W'$
%such that $(\phi a - a \phi) W = 0$ for all $a \in \mathcal{T}$.

Finally, we review three lemmas used in our study.

\begin{lem}\label{lem:inter}{\cite{terwilliger}}
For $0 \leq i,j,h \leq d$, $E_i ^* A_j E_h ^* = 0$ if and only if $p_{ij} ^h = 0$.
\end{lem}
Let $\mathcal{T}_0 (x)$ be the subspace of $\mathcal{T}(x)$ spanned by $\{ E_i ^* A_j E_h ^* \mid 0 \leq i,j,h \leq d \}$.
It is easy to see that $\mathcal{T}(x)$ is generated by $\mathcal{T}_0 (x)$ as an algebra, but in general, $\mathcal{T}_0 (x)$ may be a proper
linear subspace of $\mathcal{T}(x)$.

The following lemma was proved in \cite{terwilliger}, where the irreducible module is called the \textit{primary} module.

\begin{lem}\label{lem:primary}
In the standard module $V$ given above,

$Span \{ E_0 ^* \mathbf{1}, E_1 ^* \mathbf{1}, \dotsc , E_d ^* \mathbf{1} \}$ is an irreducible $\mathcal{T}(x)$-module of dimension $d+1$.
\end{lem}

In \cite{jaeger} an association scheme $\chi$ is called \textit{triply-regular} if
$p_{ijh}^{lmn}(x,y,z)=| R_i(x) \cap R_j(y) \cap R_h(z) |$
depends only on $i,j,h,l,m,n$ where $(x,y) \in R_l, (x,z) \in R_m$ and $(y,z) \in R_n$.

\begin{lem}\label{lem:triplyregular}{\cite{munemasa}}
Let $\chi$ be a commutative association scheme. Then $\chi$ is triply-regular if and only if $\mathcal{T}(x)=\mathcal{T}_0 (x)$ for every $x \in X$.
\end{lem}

\section{The triply-regularity of the Terwilliger algebra of $ C_{p_1} \wr C_{p_2} \wr \cdots \wr C_{p_d} $}\label{sec:main}
In this section, we will show that $ C_{p_1} \wr C_{p_2} \wr \cdots \wr C_{p_d} $
is triply-regular. Let $(X_i, \{R_j\}_{0 \leq j \leq p_i -1})$ be a cyclic scheme of order $p_i$.
We denote it by $C_{p_i}$.
In general,
we denote $X_i$ by $\{x_{i,0}, x_{i,1}, \dotsc , x_{i,{p_i -1}} \}$. We assume that $(A_1)^j = A_j$ and $R_1(x_0) = \{x_1\}$.

Let $\chi = (X, \{ R_{(i,\alpha)}\}_{ 0 \leq (i,\alpha) \leq \Sigma_{j=1} ^d (p_j - 1) })$ denote $ C_{p_1} \wr C_{p_2} \wr \cdots \wr C_{p_d} $
, where $(i,\alpha) = 0$ for $i=0$, $1 \leq \alpha \leq p_i -1$, and $(i,\alpha) = \Sigma_{j=1} ^{i-1} (p_j -1) + \alpha$ for  $1 \leq i \leq d$
, $1 \leq \alpha \leq p_i -1$.
Note that the $\alpha$ is read by modulo $p_i$.
Then $\chi$ has $X_1 \times X_2 \times \cdots \times X_d$ as the vertex set.
For $x \in X$, we denote it by $(x_1, x_2, \dotsc , x_d)$. Sometimes we specify $x$ by $(x_{1,a}, x_{2,b}, \dotsc , x_{d,z})$.

\begin{lem}\label{lem:basic1}
For $1 \leq i \leq d$, $1 \leq \alpha \leq p_i -1$ and $x \in X$, $R_{i,\alpha} (x)$ induces an association scheme which is isomorphic to $C_{p_1} \wr C_{p_2} \wr \cdots \wr C_{p_{{i-1}}}$.
\end{lem}
\begin{proof}
By definition of the wreath product, it is trivial.
\end{proof}

\begin{lem}\label{lem:basic2}
Let $y \in R_{i,\alpha}(x)$, $z \in R_{j,\beta}(x)$.
\begin{enumerate}
\item[(1)] If $i=j$, then $(y,z) \in R_h$ for some $0 \leq h \leq (i,p_{i} - 1)$.
\item[(2)] If $i > j$, then $(z,y) \in R_{i,\alpha}$.
\end{enumerate}
\end{lem}
\begin{proof}
Without loss of generality, we may assume that $x$ is $(x_{1,0}, x_{2,0}, \dotsc , x_{d,0})$.
Then $R_{i,\alpha}(x)=X_1 \times X_2 \times \cdots \times X_{i-1} \times \{ x_{i,\alpha} \} \times \{x_{(i+1),0} \} \times
\cdots \times \{ x_{d,0} \}$.
If $i=j$, then $R_{j,\beta}(x)=X_1 \times X_2 \times \cdots \times X_{i-1} \times \{ x_{i,\beta} \} \times \{x_{(i+1),0} \} \times
\cdots \times \{ x_{d,0} \}$.
So $(y,z) \in R_h$ for some $h$.
If $i > j$, then $R_{j,\beta}(x)=X_1 \times X_2 \times \cdots \times X_{j-1} \times \{ x_{j,\beta} \} \times \{x_{(j+1),0} \} \times
\cdots \times \{ x_{d,0} \}$. So $R_{j,\beta}(x) \subset X_1 \times X_2 \times \cdots \times X_{i-1} \times \{ x_{i,0} \} \times \{x_{(i+1),0} \} \times
\cdots \times \{ x_{d,0} \}$. Hence $(z,y) \in R_{i,\alpha}$.
\end{proof}

\begin{lem}\label{lem:pijh}
For $0 \leq i,j,h \leq d$ and $1 \leq \alpha \leq p_i -1$, $1 \leq \beta \leq p_j -1$, $1 \leq \gamma \leq p_h -1$,
$p_{(i,\alpha) (j,\beta)} ^{(h,\gamma)} = 0$ if and only if one of the following holds:
\begin{enumerate}
\item[(1)] $i=j=h \neq 0$ and $\alpha + \beta \not\equiv \gamma$ (mod $p_i$);
\item[(2)] $i = j < h$ , $i = h < j$ or $ j = h < i$;
\item[(3)] $h < i = j$ and $\alpha + \beta \not\equiv 0$ (mod $p_i$);
\item[(4)] $j < i =h $ and $\alpha \neq \gamma$;
\item[(5)] $i < j = h$ and $\beta \neq \gamma$;
\item[(6)] $i,j,h$ are all distinct.
\end{enumerate}
\end{lem}
\begin{proof}
(1) If $i=j=h=0$, then $p_{(i,\alpha) (j,\beta)} ^{(h,\gamma)} = 1$.
Assume that $i=j=h \neq 0$. Then for $(x,y) \in R_{i,\gamma}$, we consider $R_{i,\alpha}(x) \cap R_{i,p-\beta}(y)$.
Since $C_{p_i}$ is cyclic, $\alpha + \beta \equiv \gamma$ (mod $p_i$) if and only if $R_{i,\alpha}(x) \cap R_{i,p-\beta}(y) \neq \emptyset$.

(2) We prove only the case of $i = j < h$. For $(x,y) \in R_{h,\gamma}$, $R_{i,\alpha}(x) \cap R_{i,p-\beta}(y)=\emptyset$,
since $\bigcup_\alpha R_{i,\alpha}(x)$ and $\bigcup_\alpha R_{i,\alpha}(y)$ are disjoint.

(3) Assume that $h < i = j$. Then for $(x,y) \in R_{h,\gamma}$, $R_{i,\alpha}(x) = R_{i,\alpha}(y)$.
Thus $R_{i,\alpha}(x) \cap R_{i,p-\beta}(y) \neq \emptyset$ if and only if $\alpha + \beta \equiv 0$ (mod $p_i$).

(4) and (5) are similar to (3).
(6) is trivial.
\end{proof}

\begin{thm}\label{thm:triply}
$\chi= C_{p_1} \wr C_{p_2} \wr \cdots \wr C_{p_d}$ is triply-regular.
\end{thm}
\begin{proof}
We claim that for $(x,y) \in R_{l,a}, (x,z) \in R_{m,b}$ and $(y,z) \in R_{n,c}$, $p_{(i,\alpha) (j,\beta) (h,\gamma)}^{(l,a) (m,b) (n,c)}$ $(x,y,z)
$ $=$ $\mid R_{i,\alpha}(x) \cap R_{j,\beta}(y) \cap R_{h,\gamma}(z) \mid$ depends only on $(i,\alpha)$, $(j,\beta)$, $(h,\gamma)$, $(l,a)$, $(m,b)$, $(n,c)$.

Without loss of generality, we can assume that $l \geq m \geq n$. By Lemma \ref{lem:pijh}, it suffices to show that $p_{(i,\alpha) (j,\beta) (h,\gamma)}^{(l,a) (m,b) (n,c)}(x,y,z)$ is independent of $x,y,z$ for given $l = m \geq n$. We consider the following two cases.

Case 1: $l = m = n$.

Case 1a: $l=i$.
If either $j > i$ or $h > i$, then $p_{(i,\alpha) (j,\beta) (h,\gamma)}^{(l,a) (m,b) (n,c)}(x,y,z)=0$ by Lemma \ref{lem:pijh}.
On the other hand, if $i=h=j$, $i=h >j$ or $i=j > h$, then since the set of relations of $C_{p_i}$ is a cyclic group,
$p_{(i,\alpha) (j,\beta) (h,\gamma)}^{(l,a) (m,b) (n,c)}(x,y,z)$ is independent of $x,y,z$.

Case 1b: $l < i$.
If either $h \neq i$ or $j \neq i$, then $p_{(i,\alpha) (j,\beta) (h,\gamma)}^{(l,a) (m,b) (n,c)}(x,y,z)=0$ by Lemma \ref{lem:pijh}.
On the other hand, assume that $i=j=h$. If $\alpha \neq \beta$, $\alpha \neq \gamma$ or $\beta \neq \gamma$, then $p_{(i,\alpha) (j,\beta) (h,\gamma)}^{(l,a) (m,b) (n,c)}(x,y,z)=0$ by Lemma \ref{lem:pijh}. So we assume that $\alpha = \beta = \gamma$.
Since $l < i = j = h$, $p_{(i,\alpha) (j,\beta) (h,\gamma)}^{(l,a) (m,b) (n,c)}(x,y,z) = | R_{i,\alpha}(x) |$.

Case 1c: $i < l$.
If either $h \leq i$ or $j \leq i$, then $p_{(i,\alpha) (j,\beta) (h,\gamma)}^{(l,a) (m,b) (n,c)}(x,y,z)=0$ by Lemma \ref{lem:pijh}.
On the other hand, assume that $h=j=l$. If either $b \neq p- \gamma$ or $a \neq p -\beta$, then $p_{(i,\alpha) (j,\beta) (h,\gamma)}^{(l,a) (m,b) (n,c)}(x,y,z)=0$ by Lemma \ref{lem:pijh}. So we assume that $b = p- \gamma$ and $a = p -\beta$.
Hence $p_{(i,\alpha) (j,\beta) (h,\gamma)}^{(l,a) (m,b) (n,c)}(x,y,z)=| R_{i,\alpha}(x) | $.

Case 2: $l = m > n$.

Case 2a: $l=i$.
If either $h > l$ or $j > l$, then $p_{(i,\alpha) (j,\beta) (h,\gamma)}^{(l,a) (m,b) (n,c)}(x,y,z)=0$ by Lemma \ref{lem:pijh}.
On the other hand, assume that $l \geq h$ and $l \geq j$. Since $n < l$, by Lemma \ref{lem:pijh}, $j=h=l$ or $j=h=n$.
If $j=h=l$, then $p_{(i,\alpha) (j,\beta) (h,\gamma)}^{(l,a) (m,b) (n,c)}(x,y,z) = | R_{i,\alpha}(x) \cap R_{j,\beta}(y) |$ if $\beta = \gamma$, $0$ otherwise.
If $j=h=n$, then $p_{(i,\alpha) (j,\beta) (h,\gamma)}^{(l,a) (m,b) (n,c)}(x,y,z) = 0$ unless $c+\gamma=\beta$ (mod $p_n$).
Also $p_{(i,\alpha) (j,\beta) (h,\gamma)}^{(l,a) (m,b) (n,c)}(x,y,z) = 0$ unless $\alpha=a=b$. Thus if $c+\gamma=\beta$ (mod $p_n$) and $\alpha=a=b$,
then $p_{(i,\alpha) (j,\beta) (h,\gamma)}^{(l,a) (m,b) (n,c)}(x,y,z) = | R_{h,\gamma}(z) \cap R_{j,\beta}(y) |$.
By Lemma \ref{lem:basic1}, $p_{(i,\alpha) (j,\beta) (h,\gamma)}^{(l,a) (m,b) (n,c)}(x,y,z)$ is independent of $x, y, z$.

Case 2b: $l > i$.
If either $h \neq l$ or $j \neq l$,then $p_{(i,\alpha) (j,\beta) (h,\gamma)}^{(l,a) (m,b) (n,c)}(x,y,z)=0$ by Lemma \ref{lem:pijh}.
On the other hand, assume that $l=h=j$. Since $(y,z) \in R_{n,c}$, $p_{(i,\alpha) (j,\beta) (h,\gamma)}^{(l,a) (m,b) (n,c)}(x,y,z)=| R_{i,\alpha}(x)
\cap R_{j,\beta}(y) |$ if $\beta = \gamma$, 0 otherwise.

Case 2c: $l < i$.
If either $h \neq i$ or $j \neq i$,then $p_{(i,\alpha) (j,\beta) (h,\gamma)}^{(l,a) (m,b) (n,c)}(x,y,z)=0$ by Lemma \ref{lem:pijh}.
On the other hand, assume that $l < i=h=j$. Then $p_{(i,\alpha) (j,\beta) (h,\gamma)}^{(l,a) (m,b) (n,c)}(x,y,z)=|R_{i,\alpha}(x)|$ if $\alpha = \beta = \gamma$, 0 otherwise.
\end{proof}

\begin{thm}\label{thm:nonzero}
The complete list of nonzero triple products $E_{(i,\alpha)}^* A_{(j,\beta)} E_{(h,\gamma)}^*$, where $0 \leq i,j,h \leq d$ and
$1 \leq \alpha \leq p_i -1$, $1 \leq \beta \leq p_j -1$, $1 \leq \gamma \leq p_h -1$, is given as follows:
\begin{enumerate}
\item[(1)] $E_{(i,\alpha)}^* A_{(i,\beta)} E_{(i,\gamma)}^*$ for $i=0$;
\item[(2)] $E_{(i,\alpha)}^* A_{(i,\beta)} E_{(i,\gamma)}^*$ for $i \neq 0$ and $\alpha + \beta \equiv \gamma$ (mod $p_i$);
\item[(3)] $E_{(i,\alpha)}^* A_{(i,\beta)} E_{(h,\gamma)}^*$ for $0 \leq h < i \leq d$ and $\alpha + \beta \equiv 0$ (mod $p_i$);
\item[(4)] $E_{(i,\alpha)}^* A_{(j,\beta)} E_{(i,\gamma)}^*$ for $0 \leq j < i \leq d$ and $\alpha = \gamma$;
\item[(5)] $E_{(i,\alpha)}^* A_{(j,\beta)} E_{(j,\gamma)}^*$ for $0 \leq i < j \leq d$ and $\beta = \gamma$.
\end{enumerate}
\end{thm}
\begin{proof}
By Lemma \ref{lem:inter} and \ref{lem:pijh}, it is trivial.
\end{proof}

\section{Central primitive idempotents of the Terwilliger algebra of $ C_{p_1} \wr C_{p_2} \wr \cdots \wr C_{p_d} $}\label{sec:full}
Let $\chi = (X, \{ R_{(i,\alpha)}\}_{ 0 \leq (i,\alpha) \leq \Sigma_{j=1} ^d (p_j -1) })$ denote $ C_{p_1} \wr C_{p_2} \wr \cdots \wr C_{p_d} $
, where $(i,\alpha) = 0$ for $i=0$, $1 \leq \alpha \leq p_i -1$, and $(i,\alpha) = \Sigma_{j=1} ^{i-1} (p_j -1) + \alpha$ for  $1 \leq i \leq d$
, $1 \leq \alpha \leq p_i -1$.

In the rest of this section, we will use $[\Sigma_{i=1} ^d (p_i-1)]$ to denote the set $\{ 0, 1, \dotsc, \Sigma_{i=1} ^d (p_i -1) \}$.
We will denote $A_{(j,\beta)}$ as follows:
\begin{equation*}
\begin{bmatrix}
(A_{(j,\beta)})_{00} &(A_{(j,\beta)})_{01}  &(A_{(j,\beta)})_{02}  &\cdots  &(A_{(j,\beta)})_{0 (d, p_d - 1)}   \\
(A_{(j,\beta)})_{10} &(A_{(j,\beta)})_{11}  &(A_{(j,\beta)})_{12}  &\cdots  &(A_{(j,\beta)})_{1 (d, p_d - 1)}   \\
(A_{(j,\beta)})_{20} &(A_{(j,\beta)})_{21}  &(A_{(j,\beta)})_{22}  &\cdots  &(A_{(j,\beta)})_{2 (d, p_d - 1)}  \\
\vdots &\vdots  &\vdots  &\ddots  &\vdots   \\
(A_{(j,\beta)})_{(d, p_d - 1)0} &(A_{(j,\beta)})_{(d, p_d - 1)1}  &(A_{(j,\beta)})_{(d, p_d - 1)2}  &\cdots  &(A_{(j,\beta)})_{(d, p_d - 1)(d, p_d - 1)}   \\
\end{bmatrix},
\end{equation*}
where $((i,\alpha),(h,\xi))$-block of $A_{(j, \beta)}$ is a $n_{(i,\alpha)} \times n_{(h,\xi)}$ matrix.
In particular, for any $(i, \alpha), (j,\beta), (h, \xi) \in [\Sigma_{i=1} ^d (p_i -1)]$, we can write $E_{(i,\alpha)}^* A_{(j,\beta)} E_{(h,\xi)}^*$ as a block matrix.

Let $J_{p,q}$ denote the $p \times q$ matrix whose entries are all $1$.

\begin{lem}\label{lem:f1}
For any $(j, \beta) \in [\Sigma_{i=1} ^d (p_i -1)]$, $A_{(j, \beta)}=$
\[
\def\t{\multicolumn{1}{cl}{0}}
\left[
\begin{array}{ccccccccccccccccc}
0  &   0   &   \cdots  &    0   &    0     &  0  &\cdots  & 0& B_1    &  0   &  0  &  \cdots  &  0 &0 &0 &\cdots &0  \\
0  &   0   &   \cdots  &    0   &   0      &  0  &\cdots &0 &   B_2   &  0   &  0  &  \cdots  &  0 &0 &0 &\cdots &0  \\
\vdots   &   \vdots    &     &   \vdots    &    \vdots &  \vdots    &  \vdots   & \vdots   &  \vdots    &  \vdots  &  \vdots  &   & \vdots & \vdots    &    \vdots     &  \vdots   & \vdots     \\
0  &   0   &   \cdots  &    0   &    0  &  0  & \cdots  & 0 & B_3    & 0  &  0  &  \cdots  &  0  &0 &0 &\cdots &0 \\
0  &   0   &   \cdots  &    0   &    0     & 0 &\cdots &0 & 0&B_7  & 0  &  \cdots   &  0   &  0  & 0  &  \cdots  &0  \\
0  &   0   &   \cdots  &    0   &    0     &  0  &\cdots &0 &0&0 & B_8  & \cdots  &  0   &  0  &  0  &  \cdots  &0   \\
0  &   0   &   \cdots  &    0   &    0     &  0  &\cdots &0 &0 &0&0  &  \ddots   &  0   &  0  &  0  &  \cdots  &0   \\
0  &   0   &   \cdots  &    0   &    0     &  0  &\cdots &0 &0 &0& 0&  \cdots   & B_9  &  0  &  0  &  \cdots &0   \\
B_4  &   B_5  &   \cdots  &   B_6  &    0     &  0  & \cdots &  0   &0&  0   &  0  &  \cdots  &  0  &0 &0 &\cdots &0 \\
0  &   0   &   \cdots  &    0   &   B_{10}    &  0  & \cdots &  0   & 0& 0   &  0  &  \cdots  &  0  &0 &0 &\cdots &0 \\
0  &   0   &   \cdots  &    0   &   0    &  B_{11}  & \cdots  &  0   &  0&0   &  0  &  \cdots  &  0  &0 &0 &\cdots &0 \\
0  &   0   &   \cdots  &    0   &   0    &  0  &  \ddots & 0  &  0 &0  &  0  &  \cdots  &  0  &0 &0 &\cdots &0 \\
0  &   0   &   \cdots  &    0   &   0    &  0  & \cdots &  B_{12}  &  0&0   &  0  &  \cdots  &  0 &0 &0 &\cdots &0  \\
0  &   0   &   \cdots  &    0   &   0    &  0  & \cdots  &  0  &  0   &0&  0  &  \cdots  &  0  &B_{13} &0 &\cdots &0 \\
0  &   0   &   \cdots  &    0   &   0    &  0  & \cdots  &  0  &  0   &0&  0  &  \cdots  &  0 &0  &B_{14} &\cdots &0 \\
0  &   0   &   \cdots  &    0   &   0    &  0  & \cdots  &  0  &  0   &0 &  0  &  \cdots  &  0 &0 &0  &\ddots &0 \\
0  &   0   &   \cdots  &    0   &   0    &  0  & \cdots &  0  &  0   &0&  0  &  \cdots  &  0 &0 &0 &\cdots  & B_{15} \\
\end{array}
\right]
\]
where $B_1 = J_{1,n_{(j,\beta)}}$,
$B_2=J_{n_{(1,1)},n_{(j,\beta)}}$,
$B_3 = J_{n_{(j-1,p_{j-1}-1)},n_{(j,\beta)}}$,

$B_4 = J_{n_{(j,\beta)},1}$,
$B_5 = J_{n_{(j,\beta)},n_{(1,1)}}$,
$B_6 = J_{n_{(j,\beta)},n_{(j-1,p_{j-1}-1)}}$,

$B_7 = (A_{(j,\beta)})_{(j,1)(j,\beta+1 )}$,
$B_8 = (A_{(j,\beta)})_{(j,2)(j,\beta+2)}$,
$B_9 = (A_{(j,\beta)})_{(j,p_{j}-1- \beta)(j,p_{j}-1)}$,

$B_{10} = (A_{(j,\beta)})_{(j,p_{j}+1- \beta)(j,1)}$,
$B_{11}=(A_{(j,\beta)})_{(j,p_{j}+2- \beta)(j,2)}$,
$B_{12}=(A_{(j,\beta)})_{(j,p_{j}-1)(j,\beta-1)}$,

$B_{13}=(A_{(j,\beta)})_{(j+1,1)(j+1,1 )}$,
$B_{14}=(A_{(j,\beta)})_{(j+1,2)(j+1,2 )}$,
$B_{15}=(A_{(j,\beta)})_{(d,p_d -1)(d,p_d -1)}$.
\end{lem}
\begin{proof}
We may assume that $(j,\beta) \neq 0$.
We will consider block matrices of $A_{(j, \beta)}$.

If $i < j$, then $E_{(i,\alpha)}^* A_{(j,\beta)} E_{(h,\xi)}^* \neq 0$ if and only if $j=h$ and $\beta=\xi$ by Theorem \ref{thm:nonzero}(5).
This implies that $(A_{(j,\beta)})_{1(j,\beta)}, (A_{(j,\beta)})_{2(j,\beta)}, \dotsc, (A_{(j,\beta)})_{(j-1,p_{j-1}-1)(j,\beta)}$ are nonzero.

If $i = j$, then $E_{(i,\alpha)}^* A_{(j,\beta)} E_{(h,\xi)}^* \neq 0$ if and only if either $h=i$ and $\alpha+\beta \equiv \xi$ (mod $p_j$) or
$h<i$ and $\alpha+\beta \equiv 0$ (mod $p_j$) by Theorem \ref{thm:nonzero}(2) and \ref{thm:nonzero}(3).
This implies that $B_7, B_8, \dotsc, B_9, B_{10}, B_{11}, \dotsc, B_{12}, B_4, B_5, \dotsc, B_6$ are nonzero.

If $i > j$, then $E_{(i,\alpha)}^* A_{(j,\beta)} E_{(h,\xi)}^* \neq 0$ if and only if $h=i$ and $\alpha=\xi$ by Theorem \ref{thm:nonzero}(4).
This implies that $B_{13}, B_{14}, \dotsc, B_{15}$ are nonzero.
\end{proof}

We define the following matrices $G_{(i,\alpha)(j,\beta)}$ for all $(i,\alpha), (j,\beta) \in [\Sigma_{i=1} ^d (p_i -1)]$.
Let
\[G_{(i,\alpha)(j,\beta)} := \left\{
                    \begin{array}{ll}
                     \frac{1}{n_{(j,\beta)}} E_{(i,\alpha)}^* A_{(j,\beta)} E_{(j,\beta)}^*  & \hbox{if $i < j$}; \\
                     \frac{1}{n_{(j,\beta)}} E_{(i,\alpha)}^* A_{(i,\alpha)}^t E_{(j,\beta)}^* & \hbox{if $i > j$}; \\
                     \frac{1}{n_{(i,\alpha)}} E_{(i,\alpha)}^* J E_{(i,\beta)}^* & \hbox{if $i=j$}.
                    \end{array}
                  \right.\]
It is easy to see that the $((i,\alpha), (j,\beta))$-block of $G_{(i,\alpha)(j,\beta)}$ is $\frac{1}{n_{(j,\beta)}} J_{n_{(i,\alpha)}, n_{(j,\beta)}}$ and $\{G_{(i,\alpha)(j,\beta)} | (i,\alpha), (j,\beta) \in [\Sigma_{i=1} ^d (p_i -1)]\}$ is a linearly independent subset of $\mathcal{T}(x)$. Let $\mathcal{U}$ be the $\mathbb{C}$-subspace of $\mathcal{T}(x)$ generated by $\{G_{(i,\alpha)(j,\beta)} | (i,\alpha), (j,\beta) \in [\Sigma_{i=1} ^d (p_i -1)]\}$. Then its dimension is $(1+ \Sigma_{i=1} ^d (p_i -1))^2$.

\begin{lem}\label{lem:f2}
For any $(i,\alpha), (j,\beta), (l,\gamma), (m,\delta) \in [\Sigma_{i=1} ^d (p_i -1)]$,
\[G_{(i,\alpha)(j,\beta)} G_{(l,\gamma)(m,\delta)} = \delta_{(j,\beta)(l,\gamma)} G_{(i,\alpha)(m,\delta)}.\]
\end{lem}
\begin{proof}
It is enough to show that $G_{(i,\alpha)(j,\beta)} G_{(j,\beta)(m,\delta)} =  G_{(i,\alpha)(m,\delta)}$.
Since the $((i,\alpha),(j,\beta))$-block of $G_{(i,\alpha)(j,\beta)}$ and the $((j,\beta),(m,\delta))$-block of $G_{(j,\beta)(m,\delta)}$
are the only nonzero block, every block except for the $((i,\alpha), (m,\delta))$-block of $G_{(i,\alpha)(j,\beta)} G_{(j,\beta)(m,\delta)}$ is zero.
Thus, the $((i,\alpha), (m,\delta))$-block of $G_{(i,\alpha)(j,\beta)} G_{(j,\beta)(m,\delta)}$ is $\frac{1}{n_{(j,\beta)}} J_{n_{(i,\alpha)}, n_{(j,\beta)}}$ $\frac{1}{n_{(m,\delta)}} J_{n_{(j,\beta)}, n_{(m,\delta)}}$ $= \frac{1}{n_{(m,\delta)}} J_{n_{(i,\alpha)}, n_{(m,\delta)}}$. Thus, $G_{(i,\alpha)(j,\beta)} G_{(j,\beta)(m,\delta)} = G_{(i,\alpha)(m,\delta)}$.
\end{proof}

\begin{lem}\label{lem:f3}
Let $(h,\xi), (i,\alpha), (j,\beta) \in [\Sigma_{i=1} ^d (p_i -1)]$, the following hold.
\begin{enumerate}
\item[(1)] If $h < i$, then $A_{(h,\xi)} G_{(i,\alpha)(j,\beta)} = E_{(i,\alpha)}^* A_{(h,\xi)}$ $E_{(i,\alpha)}^* G_{(i,\alpha)(j,\beta)}$.
\item[(2)] If $h = i$ and $\alpha = \xi$, then $A_{(h,\xi)} G_{(i,\alpha)(j,\beta)} =\sum_{r=0} ^{(i-1,p_{i-1}-1)}$ $E_{r}^* A_{(i,\alpha)}$
 $E_{(i,\alpha)}^* G_{(i,\alpha)(j,\beta)}$.
\item[(3)] If $h = i$ and $\alpha \neq \xi$, then $A_{(h,\xi)} G_{(i,\alpha)(j,\beta)} = E_{(i,\rho)}^* A_{(i,\xi)} E_{(i,\alpha)}^* G_{(i,\alpha)(j,\beta)}$ for $\rho + \xi \equiv \alpha$ (mod $p_i$).
\item[(4)] If $h > i$, then $A_{(h,\xi)} G_{(i,\alpha)(j,\beta)} = E_{(h,p_h - \xi)}^* A_{(h,\xi)} E_{(i,\alpha)}^* G_{(i,\alpha)(j,\beta)}$.
\item[(5)] If $h < j$, then $G_{(i,\alpha)(j,\beta)} A_{(h,\xi)} = G_{(i,\alpha)(j,\beta)} E_{(j,\beta)}^* A_{(h,\xi)} E_{(j,\beta)}^*$.
\item[(6)] If $h = j$ and $\beta + \xi \equiv \rho$ (mod $p_j$), then $G_{(i,\alpha)(j,\beta)} A_{(h,\xi)} = G_{(i,\alpha)(j,\beta)} E_{(j,\beta)}^*$ $A_{(j,\xi)} E_{(j,\rho)}^*$.
\item[(7)] If $h = j$ and $\beta + \xi \equiv 0$ (mod $p_j$), then $G_{(i,\alpha)(j,\beta)} A_{(h,\xi)} = \sum_{r=0} ^{(j-1,p_{j-1}-1)}$ $G_{(i,\alpha)(j,\beta)} E_{(j,\beta)}^* A_{(j,\xi)} E_{r}^*$.
\item[(8)] If $h > j$, then $G_{(i,\alpha)(j,\beta)} A_{(h,\xi)} = G_{(i,\alpha)(j,\beta)} E_{(j,\beta)}^*A_{(h,\xi)}E_{(h,\xi)}^*$.
\end{enumerate}
\end{lem}
\begin{proof}
(1) By Theorem \ref{thm:nonzero}(4),
\begin{align*}
A_{(h,\xi)} G_{(i,\alpha)(j,\beta)} &= A_{(h,\xi)} E_{(i,\alpha)}^* G_{(i,\alpha)(j,\beta)}\\
 &= E_{(i,\alpha)}^* A_{(h,\xi)} E_{(i,\alpha)}^* G_{(i,\alpha)(j,\beta)}.
\end{align*}
(2) By Theorem \ref{thm:nonzero}(5),
\begin{align*}
A_{(h,\xi)} G_{(i,\alpha)(j,\beta)} &= A_{(h,\xi)} E_{(i,\alpha)}^* G_{(i,\alpha)(j,\beta)}\\
 &= \sum_{r=0} ^{(i-1,p_{i-1}-1)} E_{r}^* A_{(h,\xi)} E_{(i,\alpha)}^* G_{(i,\alpha)(j,\beta)}.
\end{align*}
(3) By Theorem \ref{thm:nonzero}(2),
\begin{align*}
A_{(h,\xi)} G_{(i,\alpha)(j,\beta)} &= A_{(h,\xi)} E_{(i,\alpha)}^* G_{(i,\alpha)(j,\beta)} \\
&= E_{(i,\rho)}^* A_{(i,\xi)} E_{(i,\alpha)}^* G_{(i,\alpha)(j,\beta)}
\end{align*}
for $\rho + \xi \equiv \alpha$ (mod $p_i$).

(4) By Theorem \ref{thm:nonzero}(3),
\begin{align*}
A_{(h,\xi)} G_{(i,\alpha)(j,\beta)} & = A_{(h,\xi)} E_{(i,\alpha)}^* G_{(i,\alpha)(j,\beta)} \\
&= E_{(h,p_h - \xi)}^* A_{(h,\xi)} E_{(i,\alpha)}^* G_{(i,\alpha)(j,\beta)}
\end{align*}

(5) By Theorem \ref{thm:nonzero}(4),
\begin{align*}
G_{(i,\alpha)(j,\beta)} A_{(h,\xi)} &= G_{(i,\alpha)(j,\beta)} E_{(j,\beta)}^* A_{(h,\xi)} \\
&= G_{(i,\alpha)(j,\beta)} E_{(j,\beta)}^* A_{(h,\xi)} E_{(j,\beta)}^*.
\end{align*}
(6) By Theorem \ref{thm:nonzero}(2),
\begin{align*}
G_{(i,\alpha)(j,\beta)} A_{(h,\xi)} &= G_{(i,\alpha)(j,\beta)} E_{(j,\beta)}^* A_{(j,\xi)} \\
&= G_{(i,\alpha)(j,\beta)} E_{(j,\beta)}^* A_{(j,\xi)} E_{(j,\rho)}^*
\end{align*}
for $\beta + \xi \equiv \rho$ (mod $p_j$).

(7) By Theorem \ref{thm:nonzero}(3),
\begin{align*}
G_{(i,\alpha)(j,\beta)} A_{(h,\xi)} &= G_{(i,\alpha)(j,\beta)} E_{(j,\beta)}^* A_{(j,\xi)} \\
&= \sum_{r=0} ^{(j-1,p_{j-1}-1)} G_{(i,\alpha)(j,\beta)} E_{(j,\beta)}^* A_{(j,\xi)} E_{r}^*.
\end{align*}

(8) By Theorem \ref{thm:nonzero}(5),
\begin{align*}
G_{(i,\alpha)(j,\beta)} A_{(h,\xi)} &=G_{(i,\alpha)(j,\beta)} E_{(j,\beta)}^*A_{(h,\xi)} \\
&= G_{(i,\alpha)(j,\beta)} E_{(j,\beta)}^*A_{(h,\xi)}E_{(h,\xi)}^*.
\end{align*}
\end{proof}

\begin{lem}\label{lem:f5}
Let $(h,\xi), (i,\alpha), (j,\beta) \in [\Sigma_{i=1} ^d (p_i -1)]$, the following hold.
\[A_{(h,\xi)} G_{(i,\alpha)(j,\beta)} = \left\{
                                         \begin{array}{ll}
                                          n_{(h,\xi)} G_{(i,\alpha)(j,\beta)}   & \hbox{if $h < i$}; \\
                                          \Sigma_{r=0} ^{(i-1,p_{i-1}-1)} n_{(i,\alpha)} G_{r(j,\beta)}    & \hbox{if $h = i$ and $\alpha = \xi$}; \\
                                          n_{(i,\alpha)} G_{(i,\alpha - \xi)(j,\beta)}    & \hbox{if $h = i$ and $\alpha \neq \xi $}; \\
                                          n_{(i,\alpha)} G_{(h,p_h -\xi)(j,\beta)}     & \hbox{if $h > i$ }. \\

                                       \end{array}
                                       \right.\]

\[G_{(i,\alpha)(j,\beta)} A_{(h,\xi)}  = \left\{
                                         \begin{array}{ll}
                                          n_{(h,\xi)} G_{(i,\alpha)(j,\beta)}   & \hbox{if $h < j$}; \\
                                          n_{(j,\rho)} G_{(i,\alpha)(j,\rho)}    & \hbox{if $h = j$ and $\beta + \xi \equiv \rho$ (mod $p_j$)}; \\
                                          \Sigma_{r=0} ^{(j-1,p_{j-1}-1)} n_r G_{(i,\alpha)r}    & \hbox{if $h = j$ and $\beta + \xi \equiv 0$ (mod $p_j$)}; \\
                                          n_{(h,\xi)} G_{(i,\alpha)(h,\xi)}     & \hbox{if $h > j$}.
                                       \end{array}
                                       \right.\]
\end{lem}
\begin{proof}
If $h < i$, then $A_{(h,\xi)} G_{(i,\alpha)(j,\beta)} = E_{(i,\alpha)}^* A_{(h,\xi)} E_{(i,\alpha)}^* G_{(i,\alpha)(j,\beta)}$ by Lemma \ref{lem:f3}(1). From Lemma \ref{lem:f1}, $((i,\alpha),(i,\alpha))$-block of $E_{(i,\alpha)}^* A_{(h,\xi)} E_{(i,\alpha)}^*$ is $(A_{(h,\xi)})_{(i,\alpha)(i,\alpha)}$ and any other block of $E_{(i,\alpha)}^* A_{(h,\xi)} E_{(i,\alpha)}^*$ is zero.
Thus, $E_{(i,\alpha)}^* A_{(h,\xi)} E_{(i,\alpha)}^* G_{(i,\alpha)(j,\beta)} = n_{(h,\xi)} G_{(i,\alpha)(j,\beta)}$.

If $h = i$ and $\alpha = \xi$, then by Lemma \ref{lem:f3}(2) and \ref{lem:f2},
\begin{align*}
A_{(h,\xi)} G_{(i,\alpha)(j,\beta)} &= \Sigma_{r=0} ^{(i-1,p_{i-1} -1)} E_r ^* A_{(i,\alpha)} E_{(i,\alpha)}^* G_{(i,\alpha)(j,\beta)}\\
 & = \Sigma_{r=0} ^{(i-1,p_{i-1} -1)} n_{(i,\alpha)} G_{r(i,\alpha)} G_{(i,\alpha)(j,\beta)} \\
 & = \Sigma_{r=0} ^{(i-1,p_{i-1} -1)} n_{(i,\alpha)} G_{r(j,\beta)}.
\end{align*}

If $h = i$ and $\alpha \neq \xi$, then by Lemma \ref{lem:f3}(3) and \ref{lem:f2},
\begin{align*}
A_{(h,\xi)} G_{(i,\alpha)(j,\beta)} & = A_{(h,\xi)} E_{(i,\alpha)}^* G_{(i,\alpha)(j,\beta)} \\
 & = E_{(i,\alpha - \xi)}^* J E_{(i,\alpha)}^* G_{(i,\alpha)(j,\beta)} \\
 & = n_{(i,\alpha)} G_{(i,\alpha - \xi)(i,\alpha)} G_{(i,\alpha)(j,\beta)} = n_{(i,\alpha)} G_{(i,\alpha - \xi)(j,\beta)}.
\end{align*}

If $h > i$, then by Lemma \ref{lem:f3}(4) and \ref{lem:f2},
\begin{align*}
A_{(h,\xi)} G_{(i,\alpha)(j,\beta)} & = A_{(h,\xi)} E_{(i,\alpha)}^* G_{(i,\alpha)(j,\beta)} \\
& = E_{(h,p_h - \xi)}^* A_{(h,\xi)} E_{(i,\alpha)}^* G_{(i,\alpha)(j,\beta)} \\
& = n_{(i,\alpha)} G_{(h,p_h - \xi)(i,\alpha)} G_{(i,\alpha)(j,\beta)} = n_{(i,\alpha)} G_{(h,p_h - \xi)(j,\beta)}.
\end{align*}

The case of $G_{(i,\alpha)(j,\beta)} A_{(h,\xi)}$ is similar.
\end{proof}

\begin{lem}\label{lem:f6}
For any $(i,\alpha), (j,\beta), (h,\xi) \in [\Sigma_{i=1} ^d (p_i -1)]$ such that $j, h \in [i-1]$,
\begin{enumerate}
\item[(1)] $E_{(i,\alpha)}^*A_{(j,\beta)} = A_{(j,\beta)}E_{(i,\alpha)}^*$.
\item[(2)] $E_{(i,\alpha)}^*A_{(j,\beta)}E_{(i,\alpha)}^*E_{(i,\alpha)}^*A_{(h,\xi)}E_{(i,\alpha)}^*=E_{(i,\alpha)}^*A_{(j,\beta)}A_{(h,\xi)}E_{(i,\alpha)}^*$.
\end{enumerate}
\end{lem}
\begin{proof}
(1) By Lemma \ref{lem:f1}, $E_{(i,\alpha)}^*A_{(j,\beta)} = E_{(i,\alpha)}^*A_{(j,\beta)}E_{(i,\alpha)}^* = A_{(j,\beta)}E_{(i,\alpha)}^*$.

(2) It is trivial.
\end{proof}

Define $F_{(i,\alpha)(h,\xi)} = \frac{1}{p_h n_{(h, \xi)}} E_{(i,\alpha)}^*(\sum_{r=0} ^{(h-1, p_{h-1}-1)} A_r + \varepsilon^\xi A_{(h,1)}$ $+ \cdots +$ $\varepsilon^{(p_h - 1)\xi}$ $A_{(h,p_h - 1)})E_{(i,\alpha)}^*$, where $2 \leq i \leq d$, $1 \leq \alpha \leq p_i - 1$, $h \in [i-1]$, $1 \leq \xi \leq p_h - 1$ and $\varepsilon = e^{\frac{2 \pi i}{p_h}}$.

\begin{thm}\label{thm:nonprimarydimen}
Let $\mathcal{T}(x)$ be the Terwilliger algebra of $C_{p_1} \wr C_{p_2} \wr \cdots \wr C_{p_d}$.
\begin{enumerate}
\item[(1)] $\mathcal{U}$ is a $\mathbb{C}$-algebra isomorphic to $M_{1+ \Sigma_{i=1} ^d (p_i -1)}(\mathbb{C})$.
\item[(2)] $\mathcal{U}$ is an ideal of $\mathcal{T}(x)$ and $\mathcal{T}(x)/\mathcal{U}$ is commutative.
\item[(3)] The set $\{ F_{(i,\alpha)(h,\xi)} \mid 2 \leq i \leq d, 1 \leq \alpha \leq p_i - 1, h \in [i-1], 1 \leq \xi \leq p_h - 1 \}$ has $\Sigma_{1 \leq h < i \leq d} (p_h - 1)(p_i - 1)$ nonzero elements and each nonzero element is a central idempotent that spans a $1$-dimensional non-primary ideal of $\mathcal{T}(x)$.
\item[(4)] $\mathcal{T}(x) \simeq M_{1+ \Sigma_{i=1} ^d (p_i -1)}(\mathbb{C}) \oplus M_1(\mathbb{C})^{\oplus \Sigma_{1 \leq h < i \leq d} (p_h - 1)(p_i - 1)}$.
\end{enumerate}
\end{thm}
\begin{proof}
(1) For any $(i,\alpha), (j,\beta), (h,\xi) \in [\Sigma_{i=1} ^d (p_i -1)]$, clearly $E_{(h,\xi)}^* G_{(i,\alpha)(j,\beta)}$ = $\delta_{(h,\xi)(i,\alpha)}$ $G_{(i,\alpha)(j,\beta)} \in \mathcal{U}$ and $G_{(i,\alpha)(j,\beta)} E_{(h,\xi)}^* = \delta_{(j,\beta)(h,\xi)}G_{(i,\alpha)(j,\beta)} \in \mathcal{U}$. Also
$A_{(h,\xi)} G_{(i,\alpha)(j,\beta)}$ $\in \mathcal{U}$ and $G_{(i,\alpha)(j,\beta)} A_{(h,\xi)} \in \mathcal{U}$ by Lemma \ref{lem:f5}.
Thus, $\mathcal{U}$ is an ideal of $\mathcal{T}(x)$. For any $(i,\alpha), (j,\beta) \in [\Sigma_{i=1} ^d (p_i -1)]$, let $e_{(i,\alpha)(j,\beta)}$ be the $(1+ \Sigma_{i=1} ^d (p_i -1)) \times (1+ \Sigma_{i=1} ^d (p_i -1))$ matrix whose $((i,\alpha), (j,\beta))$-entry is $1$ and whose other entries are all zero.
Then the linear map $\varphi: \mathcal{U} \rightarrow M_{1+ \Sigma_{i=1} ^d (p_i -1)}(\mathbb{C})$ defined by
\[ \varphi(G_{(i,\alpha)(j,\beta)}) = e_{(i,\alpha)(j,\beta)} \]
is an isomorphism by Lemma \ref{lem:f2}.

(2) By Theorem \ref{thm:triply}, $C_{p_1} \wr C_{p_2} \wr \cdots \wr C_{p_d}$ is triply-regular, namely $\mathcal{T}_0 (x) = \mathcal{T}(x)$.

\[\mathcal{T}(x)/\mathcal{U}=\text{Span} \{ E_{(i,\alpha)}^* A_{(h,\xi)} E_{(i,\alpha)}^* + \mathcal{U}  \mid  (i,\alpha) \in [\Sigma_{i=1} ^d (p_i-1) ], h \in [i-1] , 1 \leq \xi \leq p_h - 1 \}\]

For any $(i,\alpha) \in [\Sigma_{i=1} ^d (p_i -1) ]$ and $h, g < i$,
$E_{(i,\alpha)}^* A_{(h,\xi)} E_{(i,\alpha)}^* E_{(i,\alpha)}^* A_{(g,\delta)} E_{(i,\alpha)}^*=
E_{(i,\alpha)}^* A_{(g,\delta)} E_{(i,\alpha)}^* E_{(i,\alpha)}^* A_{(h,\xi)} E_{(i,\alpha)}^*$ by Lemma \ref{lem:f6}(2).

Thus, $\mathcal{T}(x)/\mathcal{U}$ is commutative.

(3) We show that $(F_{(i,\alpha)(h,\xi)})^2 = F_{(i,\alpha)(h,\xi)}$ for all $2 \leq i \leq d, 1 \leq \alpha \leq p_i - 1, h \in [i-1], 1 \leq \xi \leq p_h - 1$.

By Lemma \ref{lem:f6}, $(F_{(i,\alpha)(h,\xi)})^2 = \frac{1}{(p_h n_{(h, \xi)})^2} E_{(i,\alpha)}^*$ $(\sum_{r=0} ^{(h-1, p_{h-1}-1)}$ $p_h n_{(h, \xi)}A_r + \varepsilon^\xi p_h n_{(h, \xi)}A_{(h,1)} + \cdots + \varepsilon^{(p_i - 1)\xi} p_h n_{(h, \xi)}A_{(h,p_h - 1)})E_{(i,\alpha)}^* = F_{(i,\alpha)(h,\xi)}$.

Now we show that each element in the set $\{ F_{(i,\alpha)(h,\xi)} \mid 2 \leq i \leq d, 1 \leq \alpha \leq p_i - 1, h \in [i-1], 1 \leq \xi \leq p_h - 1 \}$ is a central idempotent that spans a $1$-dimensional ideal of $\mathcal{T}(x)$.

If $j < i$ and $j > h$, then
$A_{(j,\beta)} F_{(i,\alpha)(h,\xi)} =  \frac{1}{p_h n_{(h, \xi)}} E_{(i,\alpha)}^*A_{(j,\beta)}(\sum_{r=0} ^{(h-1, p_{h-1}-1)} A_r + \varepsilon^\xi A_{(h,1)} + \cdots + \varepsilon^{(p_h - 1)\xi} A_{(h,p_h - 1)})E_{(i,\alpha)}^*$ by Lemma \ref{lem:f6}.
Since each column sum of $(\sum_{r=0} ^{(h-1, p_{h-1}-1)} A_r + \varepsilon^\xi A_{(h,1)} + \cdots + \varepsilon^{(p_h - 1)\xi} A_{(h,p_h - 1)})$ is zero, $A_{(j,\beta)} F_{(i,\alpha)(h,\xi)} = 0$.

If $j < i$ and $j = h$, then $A_{(j,\beta)} F_{(i,\alpha)(h,\xi)} =   \frac{1}{p_h n_{(h, \xi)}} E_{(i,\alpha)}^*A_{(j,\beta)}(\sum_{r=0} ^{(h-1, p_{h-1}-1)} A_r + \varepsilon^\xi A_{(h,1)} + \cdots + \varepsilon^{(p_h - 1)\xi} A_{(h,p_h - 1)})E_{(i,\alpha)}^* = \frac{1}{p_h n_{(h, \xi)}} E_{(i,\alpha)}^*(\sum_{r=0} ^{(h-1, p_{h-1}-1)} \varepsilon^{- \beta\xi} n_{(j, 1)}A_r + \varepsilon^\xi \varepsilon^{- \beta\xi} n_{(j, 1)}A_{(h,1)} + \cdots + \varepsilon^{(p_h - 1)\xi} \varepsilon^{- \beta\xi} n_{(j, 1)}A_{(h,p_h - 1)})E_{(i,\alpha)}^* = \varepsilon^{- \beta\xi} n_{(j, 1)} F_{(i, \alpha)(h,\xi)}$.

If $j < i$ and $j < h$, then $A_{(j,\beta)} F_{(i,\alpha)(h,\xi)} =  \frac{1}{p_h n_{(h, \xi)}} E_{(i,\alpha)}^*A_{(j,\beta)}(\sum_{r=0} ^{(h-1, p_{h-1}-1)} A_r + \varepsilon^\xi A_{(h,1)} + \cdots + \varepsilon^{(p_h - 1)\xi} A_{(h,p_h - 1)})E_{(i,\alpha)}^* = \frac{1}{p_h n_{(h, \xi)}} E_{(i,\alpha)}^*(\sum_{r=0} ^{(h-1, p_{h-1}-1)} n_{(j, 1)}A_r + \varepsilon^\xi n_{(j, 1)}A_{(h,1)} + \cdots + \varepsilon^{(p_h - 1)\xi} n_{(j, 1)}A_{(h,p_h - 1)})E_{(i,\alpha)}^* = n_{(j, 1)}F_{(i,\alpha)(h,\xi)}$.

If $j \geq i$, then $A_{(j,\beta)} F_{(i,\alpha)(h,\xi)}=0$ since each column sum of $(\sum_{r=0} ^{(h-1, p_{h-1}-1)} A_r + \varepsilon^\xi A_{(h,1)} + \cdots + \varepsilon^{(p_h - 1)\xi} A_{(h,p_h - 1)})$ is zero.

Similarly, we can check the case of $F_{(i,\alpha)(h,\xi)} A_{(j,\beta)}$.

Therefore,
\[A_{(j,\beta)} F_{(i,\alpha)(h,\xi)} = F_{(i,\alpha)(h,\xi)} A_{(j,\beta)} = \left\{
                    \begin{array}{ll}
                      0                         & \hbox{if $j < i$ and $j > h$}; \\
                      \varepsilon^{- \beta\xi} n_{(j, 1)} F_{(i, \alpha)(h,\xi)}   & \hbox{if $j < i$ and $j = h$}; \\
                      n_{(j, 1)} F_{(i,\alpha)(h,\xi)}                         & \hbox{if $j < i$ and $j < h$}; \\
                      0                                                   & \hbox{if $j \geq i$}.
                    \end{array}
                     \right. \]

(4) It is trivial.
\end{proof}

\begin{cor}\label{cor:dimp}
If $p_i = p$ for $1 \leq i \leq d$, then $\mathcal{T}(x) \simeq M_{1+ d(p-1)}(\mathbb{C}) \oplus M_1(\mathbb{C})^{\oplus \frac{(d-1)d}{2}(p-1)^2}$.
\end{cor}

\begin{rem} \label{rem:1}
If $p=2$, then Corollary \ref{cor:dimp} coincides with the case that all of $n_i$ in Corollary 4.3 of \cite{song} are 2.
\end{rem}

\thanks{\textbf{Acknowledgements} This work was supported by the National Research Foundation of Korea Grant funded by the Korean Government[NRF-2010-355-C00002].}

\bibstyle{plain}

\end{document}